\documentclass[a4paper,reqno,11pt]{amsart}
\usepackage{amsfonts}
\usepackage{caption}
\usepackage{color}
\usepackage{hyperref}
\definecolor{ddorange}{rgb}{1,0.5,0}
\definecolor{ddcyan}{rgb}{0,0.2,1.0}

\newcommand{\EEE}{\color{black}}

\usepackage[normalem]{ulem}

\usepackage[left=3.5cm,right=3.5cm,top=3.5cm,bottom=3cm,headsep=0.7cm]{geometry}
\usepackage[T1]{fontenc} 
\usepackage[utf8]{inputenc}
\usepackage[french, italian, english]{babel}
\usepackage{amsmath,amssymb,amsthm,mathrsfs,esint}
\usepackage{mathtools}

\newcommand{\C}{{\mathbb C}}
\newcommand{\N}{{\mathbb N}}

\newcommand{\R}{{\mathbb R}}
\newcommand{\RR}{{\mathbb R}}
\newcommand{\Rn}{{\R}^n}

\renewcommand{\L}{{\mathcal{L}}}
\newcommand{\EE}{{\mathrm{E}}}
\renewcommand{\d}{\mathrm{d}}
\newcommand{\M}{\mathcal{M}}
\renewcommand{\H}{\mathcal{H}}
\renewcommand{\d}{\,\mathrm{d}}
\newcommand{\x}{{\times}}
\newcommand{\de}{{\partial}}
\newcommand{\sm}{\setminus}
\renewcommand{\Cap}{\mathrm{Cap}}
\renewcommand{\tilde}{\widetilde}
\renewcommand{\epsilon}{\varepsilon}

\newcommand{\Mb}{{\M_b}}
\newcommand{\Mnn}{{\mathbb{M}^{n\times n}_{sym}}}
\newcommand{\Sym}{{\mathbb{M}^{n\times n}_{sym}}}
\newcommand{\MD}{{\mathbb M}^{n{\times}n}_D}
\newcommand{\hn}{{\mathcal H}^{n-1}}

\newcommand{\HH}{{\mathcal H}} 
\newcommand{\MbD}{{\M_b(\Omega \cup \partial_D \Omega; \MD)}}
\newcommand{\Lnn}{{L^2(\Omega; \Mnn)}}
\newcommand{\wto}{\rightharpoonup}

\newcommand{\dom}{\partial \Omega}
\newcommand{\dod}{\partial_D \Omega}
\newcommand{\don}{\partial_N \Omega}
\newcommand{\ol}{\overline}

\newcommand{\wtos}{\mathrel{\mathop{\rightharpoonup}\limits^*}}

\newcommand{\xiy}{{^\xi_y}}

\newcommand{\dhn}{\,\mathrm{d}{\mathcal H}^{n-1}}

\newcommand{\MnnD}{{\mathbb{M}^{n{\times} n}_{D}}}

\newcommand{\Hu}{{W^{1,\gamma}(\Omega)}}

\newcommand{\tr}{\mathop\mathrm{tr}\,}

\newcommand{\mres}{\mathbin{\vrule height 1.6ex depth 0pt width 0.13ex\vrule height 0.13ex depth 0pt width 1.3ex}}

\theoremstyle{plain}
\begingroup
\theoremstyle{plain}
\newtheorem{theorem}{Theorem}[section]

\newtheorem{proposition}[theorem]{Proposition}
\newtheorem{lemma}[theorem]{Lemma}
\theoremstyle{definition}

\theoremstyle{remark}
\newtheorem{remark}[theorem]{Remark}

\endgroup

\makeatletter
\newcommand{\myitem}[1][]{
  \protected@edef\@currentlabel{#1}%
\item[#1]
}
\makeatother

\numberwithin{equation}{section}
\title[A lower semicontinuity result for plasticity coupled with damage]{
A  lower semicontinuity result  for linearised elasto-plasticity  coupled with damage in $W^{1,\gamma}$, $\gamma>1$}
\author{Vito Crismale}
\address{CMAP, \'Ecole Polytechnique, UMR CNRS 7641, 91128 Palaiseau Cedex, France}
\email[Vito Crismale]{vito.crismale@polytechnique.edu}

\author{Gianluca Orlando}
\address{TU M\"unchen, Zentrum Mathematik - M7, Boltzmannstrasse 3, 85747 Garching}
\email[Gianluca Orlando]{orlando@ma.tum.de}

\begin{document}
\begin{abstract}
We prove the lower semicontinuity of functionals of the form
\begin{equation*}
    \int \limits_\Omega \! V(\alpha) \d |\EE u| \, , 
\end{equation*}
with respect to the weak converge of $\alpha$ in $W^{1,\gamma}(\Omega)$, $\gamma > 1$, and the weak* convergence of $u$ in $BD(\Omega)$, where $\Omega \subset \RR^n$. These functional arise in the variational modelling of linearised elasto-plasticity coupled with damage and their lower semicontinuity is crucial in the proof of existence of quasi-static evolutions. This is the first result achieved for subcritical exponents $\gamma < n$.
\end{abstract}
\maketitle

{\small
\keywords{\textbf{Keywords:} 
lower semicontinuity, elasto-plasticity, damage, functions of bounded deformation

\bigskip
\subjclass{\textbf{MSC 2010:} 49J45, 26A45, 74C05, 74G65
}}
\setcounter{tocdepth}{1}
\tableofcontents

\section{Introduction}



Plasticity and damage play a fundamental role in material modelling for the phenomenological description of the inelastic behaviour of solids in response to applied forces. The former accounts for permanent residual deformations that persist after complete unloading and originates from the movement and the accumulation of dislocations at the microscale; the latter affects the elastic response of the material and is the result of formation of microcracks and microvoids. 

%

The coupling between plasticity and damage goes far beyond the mere theoretical interest and in fact turns out to be an effective and flexible tool that allows for the modelling of a whole spectrum of failure phenomena such as nucleation of cracks, cohesive fracture~\cite{AleMarVid14}, and fatigue under cyclic loading (see \cite[Section~3.6]{Ibr09} or \cite[Section~7.5]{Lem90}). 
These models have also attracted the attention of the mathematical community, and many recent contributions have been brought to the study of evolutionary models featuring coupling between plasticity and damage. In the quasi-static setting we mention~\cite{Cri16, CriLaz16, CriOrl18} for the case of perfect plasticity and~\cite{Cri17} for a strain-gradient plasticity model; the case of hardening for plasticity is treated in~\cite{BonRocRosTho16, RosTho17, RouVal17}, while in~\cite{RouVal16} the possible presence of damage healing is taken into account. We additionally refer to~\cite{MelScaZem19} for the study of finite-strain plasticity with damage, to~\cite{DavRouSte19} for perfect plasticity in viscoelastic solids in the dynamical setting, and to \cite{Ros19} for thermo-viscoplasticity. 

The mathematical analysis on  these  models  is  not only restricted to the proof of existence of evolutions. Motivated by the discussions in~\cite{AleMarVid14}, in~\cite{Cri16, CriLaz16} it is pointed out how the interplay between plasticity and damage leads to a mathematical formulation of the {\em fatigue} phenomenon, crucial in the description of the material behaviour under cycling loading (see also~\cite{AleCriOrl19} for fatigue in a variational model without plasticity). 
In the static setting, the strict relation between damage models with plasticity and {\em cohesive fracture} models is shown in~\cite{DMOrlToa16} through a phase-field $\Gamma$-convergence analysis in the spirit of Ambrosio-Tortorelli~\cite{AmbTor90,Iur14,ChaCri17} (cf.\ also \cite{AliBraSha99, ConFocIur16} for other phase-field approximations of cohesive energies). The previous considerations and the model presented in~\cite{AbdMarWel09} have  led  in~\cite{CriLazOrl18} to the analysis of a quasi-static evolution for a cohesive fracture model with fatigue (we also refer to~\cite{DMZan07, CagToa11, ArtCagForSol18, NegSca17} for different cohesive fracture models).


In this paper we are concerned with a lower semicontinuity problem that arises in the variational modelling of small-strain plasticity coupled with damage. In order to present the main result in this paper, we introduce some notation for damage model coupled with plasticity.


For all the details about the mathematical formulation of small-strain plasticity, we refer to~\cite{DMDSMor06}. Here we recall that the \emph{linearised strain} $\mathrm{E}u$, that is the symmetrised (spatial) gradient of the \emph{displacement} $u \colon \Omega \to \R^n$, is decomposed as the sum $\mathrm{E} u = e + p$. The {\em elastic strain} $e$ is the only term which counts for the stored elastic energy and belongs to $L^2(\Omega;\Mnn)$; the {\em plastic strain} $p$ is the variable responsible for the plastic dissipation, it describes the deformations permanent after the unloading phase, and belongs to the space $\M_b(\Omega;\MnnD)$ of bounded Radon measure with values in the space of trace-free symmetric matrices $\MnnD$. The plastic dissipation can be described according to the theory of rate-independent systems~\cite{Mie05} in terms of the so-called {\em plastic dissipation potential}, a prototypical example being given in the Von~Mises theory by 
\begin{equation} \label{eq:Von Mises}
    V \int \limits_\Omega \!  \d |p| \, ,
\end{equation}
where $V$ is a material constant and $|p|$ denotes the total variation of the measure $p$ with respect to the Euclidean (or Frobenius) norm on matrices. The constant $V$ in~\eqref{eq:Von Mises} is the radius of the ball where the trace-free part of the stress is constrained to lie during the evolution. (This {\em constraint set}, whose boundary is referred to as the {\em yield surface}, is in more general models a convex compact set in the space of trace-free symmetric matrices.)  


In presence of damage, the constraint set additionally depends on the {\em damage variable} $\alpha \colon \Omega \to [0,1]$ and the plastic dissipation potential becomes accordingly 
\begin{equation} \label{eq:damage von Mises}
    \H(\alpha, p) := \int \limits_\Omega \! V(\alpha(x)) \d |p|(x) \, , 
\end{equation}
where $V \colon [0,1] \to [m,M]$ is a continuous and nondecreasing function with $m >0$. The dependence of $V$ on $\alpha$ is one of the peculiar features of these coupled models. In gradient damage models~\cite{PhaMar10-I, PhaMar10-II, KneRosZan13}, a gradient term in the energy of the type 
\begin{equation*}
    \int \limits_\Omega \! |\nabla \alpha|^\gamma \d x \, , \quad \gamma > 1 \, ,
\end{equation*}
provides, for configurations with finite energy, a control on $\alpha$ in $W^{1,\gamma}(\Omega)$. We remark that the functional $\H$ in \eqref{eq:damage von Mises} is well defined for $\alpha \in W^{1,\gamma}(\Omega)$, $\gamma>1$, and for $p = \mathrm{E} u - e$ with $u \in BD(\Omega)$ and $e \in L^2(\Omega;\Mnn)$.   
Indeed, any $\alpha \in W^{1,\gamma}(\Omega)$ admits a precise representative~$\tilde \alpha$ defined (and uniquely determined) up to a set of $\gamma$-capacity zero, which has in particular $\H^{n-1}$-measure zero and thus it is $|p|$-negligible. For more details we refer to see Section~\ref{Sec1}.

%

In this work we study the lower semicontinuity of the dissipation potential in~\eqref{eq:damage von Mises}. Before explaining in detail our result, we present some recent developments related to this problem.


The case $\gamma>n$ has been studied in~\cite{Cri16} under very general assumptions on the plastic dissipation potential. There it is proven that functionals of the form
\begin{equation*}
    \int\limits_\Omega H\Big(\alpha(x), \frac{\d p}{\d |p|}(x)\Big) \d |p|(x) \, ,
\end{equation*}
with 
\begin{equation*}
    \text{$H$ convex, continuous, and positively one-homogeneous in the second variable}     
\end{equation*}
are lower semicontinuous with respect to the weak convergence of $\alpha$ in $W^{1,\gamma}(\Omega)$ and the weak* convergence of $p$ in $\Mb(\Omega;\Mnn)$. The proof follows from Reshetnyak's semicontinuity theorem after observing that $W^{1,\gamma}(\Omega)$ is compactly embedded in $C(\overline{\Omega})$ for $\gamma > n$. This result is the starting point for the proof of the existence of quasi-static evolutions~\cite[Theorem~4.3]{Cri16}. Unfortunately, for $n \geq 2$ the condition $\gamma > n$ precludes the application of the existence result to the case where $\alpha$ belongs to the Hilbert space $H^1(\Omega)$, often preferred in the mechanical community~\cite{PhaMar10-I, PhaMar10-II, AleMarVid15, MieHofSchAld15, AmbKruDL16, MieAldRai16, AleAmbGerDL18}.


The lower semicontinuity result has been generalised in \cite{CriOrl18} to the critical case $\gamma = n$ for plastic dissipation potentials of the type 
\begin{equation*}
    \int\limits_\Omega V(\alpha(x)) H\Big(\frac{\d p}{\d |p|}(x)\Big) \d |p|(x) \, ,
\end{equation*}
with $H$ convex and positively one-homogeneous. In spite of the failure of the compact embedding of $W^{1,n}(\Omega)$ in $C(\overline{\Omega})$, the lower semicontinuity result still holds true. The proof in~\cite{CriOrl18} is based on a concentration-compactness argument in the spirit of \cite{Lio85}, that permits to identify the dimension of the support of limits of the measures $\alpha_k \EE u_k$ for $\alpha_k$ converging weakly in $W^{1,n}(\Omega)$ and $u_k$ converging weakly* in $BD(\Omega)$. However, the technique in~\cite{CriOrl18} does not apply to the case $\gamma < n$, as shown in~\cite[Example~3.1]{CriOrl18}.


In the present work we prove a lower semicontinuity result that applies for every~$\gamma > 1$ in the special case where the plastic dissipation potential is given by~\eqref{eq:damage von Mises}, i.e., when $H$ is given by the Euclidean (or Frobenius) norm,  
assuming $V$ lower semicontinuous. We assume~$\Omega$ bounded, which is usually the case in the applications to Mechanics. The result can be generalized to the case of unbounded open sets with minor modifications.  

\begin{theorem}\label{thm:lsc with Eu}
    Let $\Omega$ be an open bounded subset of $\Rn$, 
 let $V \colon \R \to [0, +\infty]$ 
 be lower semicontinuous,  
    let $\gamma > 1$, and let $\HH$ be the functional defined in~\eqref{eq:damage von Mises}. Assume that $\alpha_k \wto \alpha$ in $W^{1,\gamma}(\Omega)$ and $u_k \wtos u$ in $BD(\Omega)$. Then
    \begin{equation}\label{2608161025}
     \HH(\alpha,\mathrm{E} u) \leq \liminf_{k\to\infty} \HH (\alpha_k,\mathrm{E}u_k) \, .
    \end{equation}
\end{theorem}


In Theorem~\ref{thm:main} below we show how Theorem~\ref{thm:lsc with Eu} implies the lower semicontinuity of~$\HH$ with respect to the weak convergence of $\alpha_k$ in $W^{1,\gamma}(\Omega)$ and the weak* convergence of~$p_k$ under the additional assumptions that~$\EE u_k = e_k + p_k$,  $u_k$ converge weakly$^*$ in~$BD(\Omega)$, and~$e_k$ converge \emph{strongly} in $\Lnn$. This lower semicontinuity result would suffice to prove the existence of quasi-static evolutions for the gradient damage models coupled with small-strain plasticity, provided one knows \emph{a priori} that  the  elastic strains $e_k$ corresponding to the discrete-time approximations of the evolution converge \emph{strongly} in $\Lnn$ (see Remark~\ref{rem:2504190838}). Obtaining such an \emph{a priori} strong convergence is possible in the case of perfect plasticity without damage~\cite{Dem09}, but unfortunately it seems a  task out of reach in the presence of damage.


The proof of Theorem~\ref{thm:lsc with Eu} is based on a slicing and localisation argument first introduced in~\cite{DMOrlToa17}. This relies on the following formula for the Euclidean norm of a symmetric $n {\times} n$ matrix $A$: 
\begin{equation*}
    |A|^2 = \sup_{(\xi^1, \dots, \xi^n)} \sum_{i=1}^n |A \xi^i \cdot \xi^i|^2 ,
\end{equation*}
where the supremum is taken over all orthonormal bases $(\xi^1,\dots,\xi^n)$ of $\RR^n$. We stress that one could conclude the semicontinuity of $\H$ only knowing the convergence (even weak) of~$e_k$ along almost any slice. Unfortunately, this is not guaranteed if~$e_k$ converge only weakly in~$\Lnn$ and this is the reason why the assumption strong convergence of~$e_k$ is needed for our proof.

\section{Notation and preliminaries}\label{Sec1}

\subsection*{Notation}
Throughout the paper we assume that $n \geq 2$. The Lebesgue measure in $\RR^n$ is denoted by $\mathcal{L}^n$, while $\H^s$ is the $s$-dimensional Hausdorff measure.

The space of $n \x n$ symmetric matrices is denoted by $\Sym$; it is endowed with the euclidean scalar product $A\!:\!B:= \tr(AB^T)$, and the corresponding euclidean norm $|A| := (A \!:\! A)^{1/2}$. The symmetrised tensor product $a \odot b$ of two vectors $a, b \in \RR^n$ is the symmetric matrix with components $(a_i b_j + a_j b_i)/2$.

\subsection*{Measures}
Let $\Omega$ be an open set in $\RR^n$. The space of bounded $\RR^m$-valued Radon measures is denoted by $\M_b(\Omega; \RR^m)$. This space can be regarded as the dual of the space $C_0(\Omega; \RR^m)$ of $\RR^m$-valued continuous functions on $\ol \Omega$ vanishing on $\de \Omega$. The notion of weak* convergence in $\M_b(\Omega;\RR^m)$ refers to this duality. Moreover, we denote by $\M^+_b(\Omega)$ the space of non-negative bounded Radon measures. If $f \in L^1(\Omega; \RR^m)$, we shall always identify the bounded Radon measure $f\mathcal{L}^n$ with the function $f$.

Let us consider a lower semicontinuous function $H \colon \Omega \x \RR^m \to [0,+\infty]$, positively $1$-homogeneous and convex in the second variable and let us consider the functional defined in accordance to the theory of convex functions of measures
\[
\int\limits_\Omega H\Big(x, \frac{\d \mu}{\d |\mu|}(x)\Big) \d |\mu|(x) \, , \quad \text{for } \mu \in \M_b(\Omega; \RR^m) \, ,
\]
where ${\d \mu}/{\d |\mu|}$ is the Radon-Nikodym derivative of $\mu$ with respect to its total variation $|\mu|$. 


We recall the classical Reshetnyak's Lower Semicontinuity Theorem \cite{Res}. For a proof we refer to \cite[Theorem 2.38]{AmbFusPal}.

\begin{theorem}[Reshetnyak's Lower Semicontinuity Theorem]\label{teo:Res}
Let $\Omega$ be an open subset of $\RR^n$. Let $\mu_k, \mu \in \M_b(\Omega; \RR^m)$. If $\mu_k \wtos \mu$ weakly* in $\M_b(\Omega; \RR^m)$, then
\begin{equation*}
\int\limits_\Omega  H\Big(x,\frac{\d \mu}{\d |\mu|}(x)\Big) \d |\mu|(x) \leq \liminf_{k \to +\infty}\int\limits_\Omega  H\Big(x,\frac{\d \mu_k}{\d |\mu_k|}(x)\Big) \d |\mu_k|(x) \, ,
\end{equation*}
for every lower semicontinuous function $H \colon \Omega \x \RR^m \to [0,+\infty]$, positively $1$-homogeneous and convex in the second variable.
\end{theorem}

\subsection*{$BV$ and $BD$ functions} 
Let $\Omega$ be an open set in $\RR^n$. 
%
%
A function $v\in L^1(\Omega)$ is a \emph{function of bounded variation} on $\Omega$, and we write $v\in BV(\Omega)$, if $\mathrm{D}_i v\in \mathcal{M}_b(U)$ for $i=1,\dots,n$, where $\mathrm{D}v=(\mathrm{D}_1 v,\dots, \mathrm{D}_n v)$ is its distributional gradient. A vector-valued function $v\colon \Omega\to \R^m$ is $BV(\Omega;\R^m)$ if $v_j\in BV(\Omega)$ for every $j=1,\dots, m$. We refer to \cite{AmbFusPal} for a detailed treatment of $BV$ functions.

For every $u \in L^1(\Omega; \RR^n)$, we denote by $\mathrm{E} u$ the $\Sym$-valued distribution on $\Omega$, whose components are given by $\mathrm{E}_{ij} u := \frac{1}{2}(\mathrm{D}_j u^i + \mathrm{D}_i u^j)$. The space $BD(\Omega)$ of \emph{functions of bounded deformation} is the space of all $u \in L^1(\Omega;\RR^n)$ such that $\mathrm{E} u \in \M_b(\Omega; \Sym)$. 

A sequence $(u_k)_k$ converges to $u$ weakly* in $BD(\Omega)$ if and only if $u_k \to u$ strongly in $L^1(\Omega;\RR^n)$ and $\mathrm{E} u_k \wtos \mathrm{E} u$ weakly* in $\M_b(\Omega; \Sym)$. We recall that for every $u \in BD(\Omega)$ the measure $\mathrm{E} u$ vanishes on sets of $\H^{n-1}$-measure zero.

%
%

We refer to the book \cite{Tem} for 
general properties of functions of bounded deformation and to \cite{AmbCosDM98} for their fine properties.

\subsection*{Capacity} For the notion of capacity we refer, e.g., to \cite{EvaGar, HeiKilMar}. We recall here the definition and some properties.  

Let $1 \leq \gamma < +\infty$ and let $\Omega$ be a  bounded,  open subset of $\RR^n$. For every subset $B \subset \Omega$, the \emph{$\gamma$-capacity} of $E$ in $\Omega$ is defined by
\[
\Cap_\gamma(E,\Omega) := \inf\Big\{ \int\limits_\Omega |\nabla \alpha|^\gamma \d x  \ : \ \alpha \in W^{1,\gamma}_0(\Omega), \ v \geq 1 \text{ a.e.\ in a neighbourhood of } E \Big\} \, .
\]
A set $E \subset \Omega$ has \emph{$\gamma$-capacity zero} if $\Cap_\gamma(E, \Omega) = 0$ (actually, the definition does not depend on the open set $\Omega$ containing $E$). A property is said to hold \emph{$\Cap_\gamma$-quasi everywhere} (abbreviated as $\Cap_\gamma$-q.e.) if it does not hold for a set of $\gamma$-capacity zero.

If $1 < \gamma \leq n$ and $E$ has $\gamma$-capacity zero, then $\H^s(E) = 0$ for every $s > n-\gamma$.

A function $\alpha \colon \Omega \to \RR$ is $\Cap_\gamma$-{\em quasicontinuous} if for every $\epsilon > 0$ there exists a set $E_\epsilon \subset \Omega$ with $\Cap_\gamma(E_\epsilon, \Omega) < \epsilon$ such that the restriction $\alpha|_{\Omega \sm E_\epsilon}$ is continuous. Note that if $\gamma > n$, a function $\alpha$ is $\Cap_\gamma$-quasicontinous if and only if it is continuous.  

Every function $\alpha \in W^{1,\gamma}(\Omega)$ admits a $\Cap_\gamma$-{\em quasicontinuous representative} $\tilde \alpha$, i.e., a $\Cap_\gamma$-quasicontinuous function $\tilde{\alpha}$ such that $\tilde \alpha = \alpha$ $\L^n$-a.e.\  in $\Omega$. The $\Cap_\gamma$-quasicontinuous representative is essentially unique, that is, if $\tilde \beta$ is another $\Cap_\gamma$-quasicontinuous representative of $\alpha$, then $\tilde \beta = \tilde{\alpha}$ $\Cap_\gamma$-q.e.\  in~$\Omega$.  It satisfies (see \cite[Theorem~4.8.1]{EvaGar})
\begin{equation} \label{eq:precise representative}
    \lim_{\rho \to 0} \frac{1}{|B_\rho(x_0)|} \int\limits_{B_\rho(x_0)} \! |\alpha(x) - \tilde \alpha(x_0)| \d x = 0 \, \quad \text{for $\Cap_\gamma$-q.e.\ $x_0 \in \Omega$} \, .
\end{equation}
 If $\alpha_k \to \alpha$ strongly in $W^{1,\gamma}(\Omega)$, then there exists a subsequence $k_j$ such that $\tilde \alpha_{k_j} \to \tilde \alpha$ $\Cap_\gamma$-q.e.\  in~$\Omega$. 

\subsection*{Slicing}
We give now some notation and recall some preliminary results about slicing. For more details, we refer the reader to \cite{AmbCosDM98}.
For every $\xi \in \mathbb{S}^{n-1}:=\{x\in \Rn \colon |x|=1\}$ and for every set $B\subset \Rn$, we define
\[
\Pi^\xi:=\{z\in \Rn \colon z\cdot\xi=0\}\quad \text{and} \quad B^\xi_y:=\{t\in \R\colon y+t \xi\in B\} \quad \text{for every }y\in \Pi^\xi\,.
\] 
For any scalar function $\alpha\colon \Omega\to \R$ and any vector function $u\colon \Omega\to \Rn$, their slices $\alpha^\xi_y\colon \Omega\xiy \to \R$ and $\widehat{u}\xiy \colon \Omega\xiy \to \R$ are defined by
\[
\alpha\xiy(t):=\alpha(y+t\xi)\quad\text{and}\quad\widehat{u}\xiy:=u(y+t\xi)\cdot \xi\,,
\]
respectively.
If $u_k$ is a sequence in $L^1(\Omega;\Rn)$ and $u\in L^1(\Omega;\Rn)$ such that $u_k\to u$ in $L^1(\Omega;\Rn)$, then for every $\xi \in \mathbb{S}^{n-1}$ there exists a subsequence $u_{k_j}$ such that 
\begin{equation}\label{2608161956}
(\widehat{u}_{k_j})\xiy\to \widehat{u}\xiy\quad\text{ in }L^1(\Omega\xiy) \quad \text{ for }\hn\text{-a.e.\ }y\in \Pi^\xi\,,
\end{equation} 
by Fubini Theorem.

Let us fix $\xi \in \mathbb{S}^{n-1}$. Let $(\mu_y)_{y \in \Pi^\xi}$ be a family of bounded 
measures in $\Omega\xiy$, such that for every Borel set $B\subset \Omega$ the map $y\mapsto \mu_y(B\xiy)$ is Borel measurable and $\hn$-integrable on~$\Pi^\xi$. Then the set function
\begin{equation}
\lambda(B)=\int\limits_{\Pi^\xi} \mu_y(B\xiy)\dhn(y)\quad\text{for all }B\subset \Omega \text{ Borel }
\end{equation}
is a measure, and we write
\[
\lambda=\int\limits_{\Pi^\xi} \mu_y \dhn(y)\quad\text{ in } \M_b(\Omega)\,.
\]
It can be seen that its total variation $|\lambda|$ is given by
\begin{equation}\label{2608161312}
|\lambda|=\int\limits_{\Pi^\xi} |\mu_y| \dhn(y)\quad\text{ in } \M_b(\Omega)\,.
\end{equation}
A function $u\in L^1(\Omega;\Rn)$ belongs to $BD(\Omega)$ if and only if for every direction $\xi \in \mathbb{S}^{n-1}$ (or, equivalently, for any $\xi$ of the form $\xi_i +\xi_j$, $i,\,j=1, \dots, n$ for a fixed basis  $\{\xi_1,\dots,\xi_n \}$ of $\Rn$)
\[
\widehat{u}\xiy\in BV(\Omega\xiy) \text{ for }\hn\text{-a.e.\ } y\in \Pi^\xi \quad\text{and}\quad \int\limits_{\Pi^\xi}|\mathrm{D}\widehat{u}\xiy|(\Omega\xiy)\dhn(y) < +\infty\,.
\]
Moreover, if $u\in BD(\Omega)$ then for every $\xi \in \mathrm{S}^{n-1}$ it holds that
\begin{equation*}
\mathrm{E}u\,\xi\cdot \xi=\int\limits_{\Pi^\xi} \mathrm{D}\widehat{u}\xiy\dhn(y)\,\quad\text{ in } \M_b(\Omega)\,.
\end{equation*} 
In particular, by \eqref{2608161312}, we have that
\begin{equation}\label{2608161858}
|\mathrm{E}u\,\xi\cdot \xi|=\int\limits_{\Pi^\xi} |\mathrm{D}\widehat{u}\xiy|\dhn(y)\,\quad\text{ in } \M_b(\Omega)\,.
\end{equation}
Let $\alpha \in L^1(\Omega)$ and $\gamma\in [1,\infty)$. Then $\alpha\in W^{1,\gamma}(\Omega)$ if and only if for every $\xi \in \mathbb{S}^{n-1}$
\[
\alpha\xiy\in W^{1,\gamma}(\Omega\xiy) \text{ for }\hn\text{-a.e.\ } y\in \Pi^\xi \quad\text{and}\quad \int\limits_{\Pi^\xi}\Big(\int\limits_{\Omega\xiy}|\nabla\alpha\xiy(t)|^\gamma\d t\Big)\dhn(y) < +\infty\,.
\]
 If $\alpha\in W^{1,\gamma}(\Omega)$ then for every $\xi \in \mathrm{S}^{n-1}$ it holds that
\begin{equation}\label{2608162116}
\int\limits_\Omega |\nabla\alpha \cdot \xi|^\gamma \d x=\int\limits_{\Pi^\xi}\Big(\int\limits_{\Omega\xiy}|\nabla\alpha\xiy(t)|^\gamma\d t\Big)\dhn(y)\,.
\end{equation} 
Moreover, $ (\nabla\alpha \cdot \xi)\xiy =\nabla\alpha\xiy$ for $\hn$-a.e.\ $y\in \Pi^\xi$. 
%
\begin{remark} \label{rmk:good representative}
    Let $\alpha \in W^{1,\gamma}(\Omega)$. Then the slice $\widetilde{\alpha}\xiy$ of the $\Cap_\gamma$-quasicontinuous representative~$\tilde \alpha$ of~$\alpha$ is the continuous representative in the equivalence class of $\alpha\xiy$ for $\H^{n-1}$-a.e.\ $y\in \Pi^\xi$. Indeed, $\tilde \alpha$ is the precise representative of $\alpha$ in the sense of~\eqref{eq:precise representative}. By~\cite[Theorem~3.108]{AmbFusPal} it follows that, for $\H^{n-1}$-a.e.\ $y\in \Pi^\xi$, $\tilde \alpha^\xi_y$ is a good representative of $\alpha^\xi_y$, i.e., its pointwise total variation coincides with $|\mathrm{D} \alpha^\xi_y|(\Omega^\xi_y)$. We conclude that $\tilde \alpha^\xi_y$ is continuous by~\cite[Theorem~3.28]{AmbFusPal}.
\end{remark}

\subsection*{Auxiliary results} The proof of Theorem~\ref{thm:lsc with Eu} employs some techniques developed for the proof of \cite[Theorem~4.1]{DMOrlToa17}.  We will use the following well-known formula for the Euclidean norm of symmetric matrices (for a proof cf., e.g., Proposition~\ref{2608161807}). 
\begin{proposition}\label{2608161807}
For every $A\in \Mnn$ we have
\begin{equation*}\label{2608161808}
|A|=\sup_{(\xi^1,\dots,\xi^n)}\bigg(\sum_{i=1}^n |A \xi^i\cdot \xi^i|^2\bigg)^{\!\!1/2} ,
\end{equation*}
where the supremum is taken over all orthonormal bases $(\xi^1,\dots,\xi^n)$ of $\Rn$, or, equivalently, over the columns of all rotations $R\in O(n)$.
\end{proposition} 
We recall also the following localization lemma. We refer to~\cite[Lemma~15.2]{Bra} for its proof.
\begin{lemma}\label{2608161816}
Let $\Lambda$ be a function defined on the family of open subsets of $\Omega$, which is superadditive on open sets with disjoint compact closure. Let $\lambda$ be a positive measure on $\Omega$, and let $\varphi_j$, $j\in\N$, be nonnegative Borel functions such that
\begin{equation*}
\int\limits_K \varphi_j\d \lambda\leq \Lambda(A)
\end{equation*}
for every open set $A\subset \Omega$, for every compact set $K\subset A$, and for every $j\in \N$. Then 
\begin{equation*}
    \int\limits_K\sup_j \varphi_j \d\lambda \leq \Lambda(A)
\end{equation*}
for every open set $A\subset \Omega$ and for every compact set $K\subset A$. Moreover, if $A$ is an open set such that $\Lambda(A) < +\infty$, then 
\begin{equation*}
\int\limits_K\sup_j \varphi_j \d\lambda=\sup\Big\{\sum_{j=1}^r \int\limits_{K_j} \varphi_j \d\lambda \colon (K_j)_{j=1}^r \text{ disjoint compact subsets of }K, r\in \N\Big\} 
\end{equation*}
for every compact set $K\subset A$.
\end{lemma}

\section{The lower semicontinuity theorem}\label{3oSec8}

In this section we let $\Omega$ be an open bounded
subset of $\Rn$, $n\geq 2$, $V \colon \R \to [0,+\infty]$ be lower semicontinuous,  and we fix $\gamma > 1$. The starting point of the proof of Theorem~\ref{thm:lsc with Eu} is the following lower bound: given a direction $\xi \in \mathbb{S}^1$, for every  $\alpha \in W^{1,\gamma}(\Omega)$ 
and $u \in BD(\Omega)$ we have that 
\begin{equation*}
    \H(\alpha, \EE u) = \int\limits_{\Omega} V(\widetilde{\alpha})\,\mathrm{d}|\mathrm{E}u| \geq \int\limits_\Omega V(\widetilde{\alpha}) \d|\mathrm{E}u\,\xi\cdot\xi| \, .
\end{equation*}
 In the previous formula $|\cdot|$ denotes  the Euclidean norm (or Frobenius norm) of a matrix
and $\widetilde{\alpha}$ is the $\Cap_\gamma$-quasicontinuous representative of~$\alpha$.
Notice that the definition of $\HH$ is well posed, since
 $\widetilde{\alpha}$ is defined at $\hn$-a.e.\ $x\in\Omega$ and the measure $\mathrm{E}u$ does not charge sets of dimension less than $n-1$.

For this reason it is convenient to introduce the functionals $\mathcal{F}_\xi$, defined for every direction $\xi\in \mathbb{S}^{n-1}$ as follows: for every 
$\alpha\in W^{1,\gamma}(\Omega)$,
$u\in BD(\Omega)$, and $A\subset \Omega$ open, we put
\begin{equation}\label{2608162103}
\mathcal{F}_\xi(\alpha,u;A):=\int\limits_A V(\widetilde{\alpha}) \d|\mathrm{E}u\,\xi\cdot\xi|=\int\limits_{\Pi^\xi}\Big( \int\limits_{A\xiy} V(\widetilde{\alpha}\xiy(t)) \d |\mathrm{D}\widehat{u}\xiy|(t) \Big) \d \hn(y)\,.
\end{equation}
Notice that the second equality in the formula above follows from \eqref{2608161858}.

We first prove the lower semicontinuity of these functionals, and then we deduce Theorem~\ref{thm:lsc with Eu} using Proposition~\ref{2608161807} and Lemma~\ref{2608161816}.
\begin{proposition}\label{2608161903}
Let $\xi \in \mathbb{S}^{n-1}$ and let $\alpha_k$, 
$\alpha \in W^{1,\gamma}(\Omega)$, 
$u_k$, $u\in BD(\Omega)$ be such that $\alpha_k\wto\alpha$ in $\Hu$ and $u_k\wtos u$ in $BD(\Omega)$. Then
\begin{equation}\label{2608162001}
\mathcal{F}_\xi(\alpha,u;A)\leq \liminf_{k\to \infty} \mathcal{F}_\xi (\alpha_k, u_k; A)
\end{equation}
for every open set $A\subset \Omega$.
\end{proposition}
\begin{proof}
Let $\xi\in \mathbb{S}^{n-1}$, $A\subset \Omega$ open, $\alpha_k\wto\alpha$ in $\Hu$, and $u_k\wtos u$ in $BD(\Omega)$. Let us fix $\varepsilon>0$. By \eqref{2608161956}, upon extracting a (not relabeled) subsequence, we deduce that,  for $\hn\text{-a.e.\ }y\in \Pi^\xi$, 
\begin{equation}\label{2608162042}
(\widetilde{\alpha}_k)\xiy\to \widetilde{\alpha}\xiy\,,\quad (\widehat{u}_{k})\xiy\to \widehat{u}\xiy \quad \text{in }L^1(\Omega\xiy)\,,
\end{equation}
 and that the liminf in \eqref{2608162001} (that we may assume finite) is actually a limit.

We claim that for $\hn\text{-a.e.\ }y\in \Pi^\xi$
\begin{equation}\label{2608162018}
\int\limits_{A\xiy}  V(\widetilde{\alpha}\xiy)  \d|\mathrm{D} \widehat{u}\xiy \EEE|\leq \liminf_{k\to \infty}\bigg(\int\limits_{A\xiy} \big( V((\widetilde{\alpha}_k)\xiy) + \varepsilon  \big) \d|\mathrm{D}(\widehat{u}_k)\xiy| + \varepsilon \int\limits_{A\xiy}|\nabla(\widetilde{\alpha}_k)\xiy|^\gamma\,\d t \bigg)\,.
\end{equation}
To prove the claim, we start by observing that the boundedness of $\alpha_k$ in $W^{1,\gamma}(\Omega)$ and of $u_k$ in $BD(\Omega)$  implies
\begin{equation*}
    \begin{split}
        +\infty & > \liminf_{k \to \infty} \big[ \mathcal{F}_\xi(\alpha_k, u_k;A) +  \epsilon |\mathrm{E} u_k \xi \cdot \xi|(A)  + \epsilon\|\nabla \alpha_k \cdot \xi \|_{L^\gamma(A) }^\gamma \big] \\
        & = \liminf_{k \to \infty} \int_{\Pi^\xi} \bigg( \int\limits_{A\xiy} \big( V((\widetilde{\alpha}_k)\xiy(t))  + \epsilon  \big) \d|\mathrm{D}(\widehat{u}_k)\xiy|(t)+\varepsilon \int\limits_{A\xiy}|\nabla(\widetilde{\alpha}_k)\xiy(t)|^\gamma\,\d t \bigg) \d \H^{n-1}(y) \\
        & \geq  \int_{\Pi^\xi} \liminf_{k \to \infty} \bigg( \int\limits_{A\xiy} \big( V((\widetilde{\alpha}_k)\xiy(t))  + \epsilon  \big)\d|\mathrm{D}(\widehat{u}_k)\xiy|(t)+\varepsilon \int\limits_{A\xiy}|\nabla(\widetilde{\alpha}_k)\xiy(t)|^\gamma\,\d t \bigg) \d \H^{n-1}(y) \, ,
    \end{split}
\end{equation*}
where in the equality we applied~\eqref{2608162103} and Fubini's Theorem, while the last inequality follows from Fatou's Lemma.  
From the previous inequality it follows that for $\H^{n-1}$-a.e.\ $y\in \Pi^\xi$ 
\begin{equation}  \label{eq:liminf of section finite}
    \liminf_{k \to \infty} \bigg( \int\limits_{A\xiy} \big( V((\widetilde{\alpha}_k)\xiy)  + \epsilon   \big) \d|\mathrm{D}(\widehat{u}_k)\xiy|+\varepsilon \int\limits_{A\xiy}|\nabla(\widetilde{\alpha}_k)\xiy|^\gamma\,\d t \bigg) < +\infty \, .
\end{equation}
Moreover we remark that for $\H^{n-1}$-a.e.\ $y \in \Pi^\xi$ we have that $(\tilde \alpha_k)^\xi_y$ is the continuous representative in the equivalence class of $(\alpha_k)^\xi_y$ for every $k$ and $\tilde \alpha^\xi_y$ is the continuous representative in the equivalence class of $ \alpha^\xi_y$.\footnote{Indeed, let $N_k := \{ y \in \Pi^\xi \colon \text{$(\tilde \alpha_k)^\xi_y$ is not the continuous representative  of $(\alpha_k)^\xi_y$} \}$. By Remark~\ref{rmk:good representative} we have that $\H^{n-1}(N_k) = 0$. The set $N := \bigcup_k N_k$ satisfies $\H^{n-1}(N) = 0$ and for every $y \in \Pi^\xi \sm N$ we have that $(\tilde \alpha_k)^\xi_y$ is the continuous representative in the equivalence class of $(\alpha_k)^\xi_y$ for every $k$.}
 Let us fix~$y \in \Pi^\xi$ that satisfies this last property and~\eqref{2608162042}, \eqref{eq:liminf of section finite}. We extract a subsequence~$k_j$, possibly depending on $y$, such that 
 \begin{equation}  \label{eq:liminf of section is limit}
    \begin{split}
        & \lim_{j \to \infty} \bigg( \int\limits_{A\xiy} \big( V((\widetilde{\alpha}_{k_j})\xiy)   + \epsilon   \big)\d|\mathrm{D}(\widehat{u}_{k_j})\xiy|+\varepsilon \int\limits_{A\xiy}|\nabla(\widetilde{\alpha}_{k_j})\xiy|^\gamma\,\d t \bigg) \\
        &  \quad = \liminf_{k \to \infty} \bigg( \int\limits_{A\xiy} \big( V((\widetilde{\alpha}_k)\xiy)  + \epsilon   \big)\d|\mathrm{D}(\widehat{u}_k)\xiy|+\varepsilon \int\limits_{A\xiy}|\nabla(\widetilde{\alpha}_k)\xiy|^\gamma\,\d t \bigg) < +\infty \, .
    \end{split}
\end{equation}
Since $\varepsilon$ is fixed, the sequences $(\widehat{u}_{k_j})\xiy$ and  $(\widetilde{\alpha}_{k_j})\xiy$ are bounded in $BV(\Omega\xiy)$ and $W^{1,\gamma}(\Omega\xiy)$, respectively.  Together with~\eqref{2608162042}, \EEE this implies that 
  \begin{equation*}
(\widetilde{\alpha}_{k_j})\xiy  \wto  \widetilde{\alpha}\xiy \quad\text{in }W^{1,\gamma}(\Omega\xiy)\,,\quad  (\widehat{u}_{k_j})\xiy   \wtos   \widehat{u}\xiy\quad \text{in }BV(\Omega\xiy)\,.
\end{equation*}
Recalling that  $(\tilde \alpha_k)^\xi_y$ is the continuous representative of $(\alpha_k)^\xi_y$ for every $k$ and $\tilde \alpha^\xi_y$ is the continuous representative of $ \alpha^\xi_y$  we deduce that $(\widetilde{\alpha}_{k_j})\xiy \to \widetilde{\alpha}\xiy$ uniformly. \EEE Applying Theorem~\ref{teo:Res}, we deduce that
\begin{equation*}
\begin{split}
\int\limits_{A\xiy} V(\widetilde{\alpha}\xiy)\d|\mathrm{D}\widehat{u}\xiy|&\leq \liminf_{j\to \infty} \int\limits_{A\xiy} V((\widetilde{\alpha}_{k_j})\xiy)\d|\mathrm{D}(\widehat{u}_{k_j})\xiy| \\ &\leq \liminf_{k\to \infty}   \Bigg[\int\limits_{A\xiy} ( V((\widetilde{\alpha}_{k})\xiy)  + \epsilon  \big)\d|\mathrm{D}(\widehat{u}_{k})\xiy|+\varepsilon \int\limits_{A\xiy}|\nabla(\widetilde{\alpha}_{k})\xiy|^\gamma\,\d t \Bigg]\,.
\end{split}
\end{equation*}
This concludes the proof of the claim in~\eqref{2608162018}.

Integrating~\eqref{2608162018} with respect to $y \in \Pi^\xi$ and recalling \eqref{2608162103} and \eqref{2608162116}, we deduce by Fatou Lemma that
\begin{equation*}
\mathcal{F}_\xi(\alpha,u;A)\leq \liminf_{k\to \infty} \mathcal{F}_\xi(\alpha_k,u_k;A) +  \epsilon \limsup_{k \to \infty} |\mathrm{E} u_k \xi \cdot \xi|(A)   + \varepsilon \limsup_{k\to \infty} \int\limits_A |\nabla \alpha_k \cdot \xi|^\gamma \,\d x\,. 
\end{equation*}
Since the sequence $\alpha_k$ is bounded in $W^{1,\gamma}(\Omega)$,   $u_k$ is bounded in $BD(\Omega)$,   and $\varepsilon$ is arbitrary, the proof is concluded.
\end{proof}

We are now ready to prove the main result.

\begin{proof}[Proof of Theorem~\ref{thm:lsc with Eu}]
Let $(\xi^1,\dots,\xi^n)$ be an orthonormal basis of $\Rn$, and let us prove first that, for every 
$\alpha\in   W^{1,\gamma}(\Omega)$,  
$u\in BD(\Omega)$, and $A\subset \Omega$ open, it holds
\begin{equation}\label{2608162344}
\bigg(\sum_{i=1}^n\mathcal{F}_{\xi^i}(\alpha,u;A)^2\bigg)^{\!\! 1/2} \! \leq \int\limits_A V(\widetilde{\alpha})\d|\mathrm{E}u|  \,.
\end{equation}
Indeed, by H\"older's Inequality with respect to the measure $V(\tilde \alpha) |\mathrm{E}u|$ we get that
\begin{equation*}
\begin{split}
\mathcal{F}_{\xi^i}(\alpha,u;A)^2&=\bigg(\, \int\limits_A V(\widetilde{\alpha})\bigg| \frac{\mathrm{d} \mathrm{E}u}{\mathrm{d} |\mathrm{E}u|}\xi^i\cdot \xi^i\bigg| \d |\mathrm{E}u|\bigg)^{\!\! 2}\\&\leq \bigg(\, \int\limits_A V(\widetilde{\alpha})\bigg| \frac{\mathrm{d} \mathrm{E}u}{\mathrm{d}  |\mathrm{E}u|}\xi^i\cdot \xi^i\bigg|^2 \d |\mathrm{E}u|\bigg) \int\limits_A V(\widetilde{\alpha}) \d |\mathrm{E}u|\,.
\end{split}
\end{equation*}
Summing for $i=1,\dots,n$,  we obtain that
\begin{equation*}
\begin{split}
\bigg(\sum_{i=1}^n\mathcal{F}_{\xi^i}(\alpha,u;A)^2\bigg)^{\!\! 1/2} &\leq \bigg(\int\limits_A V(\widetilde{\alpha})\sum_{i=1}^n\bigg| \frac{\mathrm{d} \mathrm{E}u}{\mathrm{d}  |\mathrm{E}u|}\xi^i\cdot \xi^i\bigg|^2 \d |\mathrm{E}u|\bigg)^{\!\! 1/2}  \bigg(\int\limits_A V(\widetilde{\alpha}) \d |\mathrm{E}u|\bigg)^{\!\! 1/2}\hspace{-1em}\\&\leq \int\limits_A V(\widetilde{\alpha}) \d |\mathrm{E}u|\,,
\end{split}
\end{equation*}
and thus \eqref{2608162344} is proven. Notice that in the last inequality above we have used   Proposition~\ref{2608161808}   and the fact that 
\begin{equation}\label{2908162106}
\bigg| \frac{\mathrm{d} \mathrm{E}u}{\mathrm{d}  |\mathrm{E}u|}(x)\bigg|=1\quad\text{for }|\mathrm{E}u|\text{-a.e.\ }x\in \Omega\,.
\end{equation}

Let $\alpha_k,\,\alpha \in   W^{1,\gamma}(\Omega)$,  
$u_k$, $u\in BD(\Omega)$ such that $\alpha_k\wto\alpha$ in $\Hu$ and $u_k\wtos u$ in~$BD(\Omega)$. Let us prove \eqref{2608161025}.
Let $\Lambda$ be the function defined on every open set $A\subset\Omega$ by
\begin{equation} \label{eq:def of Lambda}
\Lambda(A):=\liminf_{k\to \infty} \int\limits_A V(\widetilde{\alpha}_k) \d |\mathrm{E}u_k|\,.
\end{equation}
Moreover, let $R_j$ be a sequence dense in $O(n)$ and let $\xi^1_j,\dots,\xi^n_j$ be the column vectors of~$R_j$. Let us
define the vector functions $\varphi_j=(\varphi_j^1,\dots,\varphi_j^n)$ by putting
\begin{equation}\label{2908162105}
\varphi_j^i(x):=V(\widetilde{\alpha}(x))\bigg|\frac{\mathrm{d} \mathrm{E}u}{\mathrm{d}  |\mathrm{E}u|}(x)\,\xi^i_j\cdot \xi^i_j\bigg|\quad\text{for every }j\in\N,\,i=1,\dots,n,\,\text{and }x\in\Omega\,.
\end{equation}
Recalling \eqref{2608162103}, it holds that for every $j\in \N$ and $A\subset \Omega$ open  
\begin{equation}\label{2908162046}
\bigg|\int\limits_A\varphi_j\d|\mathrm{E}u|\bigg|=  \bigg(\sum_{i=1}^n \bigg( \int\limits_A\varphi_j^i\d|\mathrm{E}u| \bigg)^{\!\!2} \, \bigg)^{\!\! 1/2}  = \bigg(\sum_{i=1}^n\mathcal{F}_{\xi_j^i}(\alpha,u;A)^2\bigg)^{\!\! 1/2}.
\end{equation}
By Proposition~\ref{2608161903}, for every $j\in\N$, $i=1,\dots,n$, and $A\subset \Omega$ open, we have that
\begin{equation*}
\mathcal{F}_{\xi_j^i}(\alpha,u;A)\leq \liminf_{k\to \infty} \mathcal{F}_{\xi_j^i} (\alpha_k, u_k; A)\,,
\end{equation*} 
and then, by the superadditivity of the liminf, it follows that
\begin{equation*}
\bigg(\sum_{i=1}^n\mathcal{F}_{\xi_j^i}(\alpha,u;A)^2\bigg)^{\!\!1/2} \leq \liminf_{k\to \infty} \bigg(\sum_{i=1}^n\mathcal{F}_{\xi_j^i} (\alpha_k, u_k; A)^2\bigg)^{\!\!1/2},
\end{equation*} 
By the previous inequality, \eqref{2608162344}, \eqref{eq:def of Lambda}, and \eqref{2908162046} we obtain that 
\begin{equation}
\bigg|\int\limits_A\varphi_j\d|\mathrm{E}u|\bigg|\leq \Lambda(A)\,.
\end{equation}
Using the superadditivity of $\Lambda$, we infer that
\begin{equation*}
\begin{split}
\int\limits_K|\varphi_j|\d|\mathrm{E}u|&=  \sup\bigg\{ \sum_{h=1}^r\bigg|\int\limits_{B^h}\varphi_j\d|\mathrm{E}u|\bigg|\colon(B^h)_{h=1}^r \text{ disjoint Borel subsets of }K,\,r\in\N\bigg\}  \\ 
&=\sup\bigg\{ \sum_{h=1}^r\bigg|\int\limits_{K^h}\varphi_j\d|\mathrm{E}u|\bigg|\colon(K^h)_{h=1}^r \text{ disjoint compact subsets of }K,\,r\in\N\bigg\}
\\&\leq \sup\bigg\{\sum_{h=1}^r\Lambda(A^h)\colon(A^h)_{h=1}^r \, , \  A^h \subset A \text{ with disjoint compact closure},  \,r\in\N\bigg\} \\
&\leq \Lambda(A)
\end{split}
\end{equation*}
for every compact set $K$ and for every open set $A$ such that $K\subset A \subset \Omega$.
Lemma~\ref{2608161816} gives that
\begin{equation}
\int\limits_K\sup_{j}|\varphi_j|\d|\mathrm{E}u|\leq \Lambda(A)\,.
\end{equation}
By~\eqref{2908162106}, \eqref{2908162105}, and Proposition~\ref{2608161807} we deduce that 
\begin{equation*}
    \sup_{j}|\varphi_j| = V(\tilde \alpha)
\end{equation*}
and therefore
\begin{equation*}
\int\limits_KV(\widetilde{\alpha})\d|\mathrm{E}u|\leq \Lambda(A)\,,
\end{equation*}
for every compact set $K$ such that $K\subset A$.
We conclude the proof by the arbitrariness of~$K$ and by recalling the definition of $\Lambda$ in~\eqref{eq:def of Lambda}.
\end{proof}

\begin{remark}
    The proof of Theorem~\ref{thm:lsc with Eu} also works in different settings, e.g., in the case where the plastic potential is defined through a convex and positively one-homogeneous function $H \colon \Mnn \to [0,+\infty)$ which satisfies 
    \begin{equation*}
        H(A)^q =  \sup_{(\xi_1, \dots, \xi_n)} \sum_{i=1}^n |A \xi^i \cdot \xi^i|^q , \quad q \in (1,\infty) \, ,
    \end{equation*}
    where the supremum is taken over all orthonormal bases $(\xi^1,\dots,\xi^n)$ of $\RR^n$. Such matrix norms $H$ are usually referred to as $q$-Schatten norm, cf.~\cite{HorJoh}.
\end{remark}



In the remaining part of this section we show under which assumptions the technique in the proof of Theorem~\ref{thm:lsc with Eu} can be adapted to prove the lower semicontinuity of the plastic potential $\H$ introduced in~\eqref{eq:damage von Mises}. 
We consider here a slight generalisation, where we allow the plastic strain $p$ to charge some part of $\dom$, the boundary of $\Omega$.

Let us assume that the boundary of $\Omega$
is  Lipschitz and partitioned as 
\[
\dom=\dod\cup \don\cup N\,,
\] 
with $\dod$ and $\don$  relatively open, $\dod \cap \don =\emptyset$, $\mathcal{H}^{n-1}(N)=0$, and
$\dod \neq \emptyset$. 
A~boundary datum $w \in H^1(\Omega;\Rn)$ will be suitably imposed on the Dirichlet boundary $\dod$.

We consider from now on the functional $\H$ as defined as in perfect plasticity with damage, where it represents the \emph{plastic potential}. 
This is defined on the class of admissible~$p$ defined as follows. We introduce the set of admissible triples of \emph{displacement}, \emph{elastic strain}, and \emph{plastic strain} for the boundary datum $w$,
\begin{equation*}
\begin{split}
\mathcal{A}(w):=\{(u,e,p) \in  BD(\Omega)&  \x \Lnn \x \MbD : \\ & \mathrm{E}u=e + p \,\text{ in }\Omega\, ,\, p \mres \dod = (w-u) \odot \nu \, \mathcal{H}^{n-1}\mres \dod\}\, .
\end{split}
\end{equation*}
A plastic strain $p$ is \emph{admissible} (for $w$) if it belongs to
\begin{equation*} 
\Pi(\Omega):=\{ p \in \MbD \colon \text{there exist } u,\, e \text{ such that } (u, e, p) \in \mathcal{A}(w) \}\,.
\end{equation*}

The functional $\H$ is then defined on 
 $W^{1,\gamma}(\Omega) \EEE
\times \Pi(\Omega)$   by 
\begin{equation} \label{eq:precise H}
\H(\alpha, p):=\hspace{-0.7em}\int\limits_{\Omega \cup \dod}\hspace{-0.7em} V(\tilde \alpha(x)) \, \d |p|(x) \,.
\end{equation}
We now prove the claimed lower semicontinuity result.  We stress that a crucial assumption for the validity of our proof is the strong convergence of the elastic strain. Up to our knowledge, a proof under the sole assumption of weak convergence of the elastic strain is still missing. 

\begin{theorem}\label{thm:main}
     Let $\Omega$ be an open bounded Lipschitz subset of $\Rn$, $V \colon \R \to [0, +\infty]$ be lower semicontinuous,   and let $\gamma >1$. Let $\alpha_k$, $\alpha\in   W^{1,\gamma}(\Omega)$   and $(u_k,e_k,p_k)$, $(u,e,p)\in \mathcal{A}(w)$ be such that $\alpha_k\wto \alpha$ in $\Hu$, $u_k\wtos u$ in $BD(\Omega)$, and $e_k\to e$ strongly in $\Lnn$. Then
    \begin{equation*}
    \H(\alpha,p)\leq \liminf_{k\to \infty} \H(\alpha_k, p_k)\,.
    \end{equation*}
    \end{theorem} 
\begin{proof}
Let $\widetilde{\Omega}$ be a smooth open set such that $\Omega \cup \dod\subset \widetilde{\Omega}$ and $\dom \cap \widetilde{\Omega}=\dod$, and let us define, for every $(u_k,e_k,p_k)$, $(u,e,p)$ as in the assumptions of the theorem, the extended functions
\begin{equation*} 
\ol u_k:=
\begin{cases}
u_k\,\,&\text{in }\Omega \, ,\\
w\,\,&\text{in }\widetilde{\Omega} \setminus\Omega \, ,
\end{cases}
\quad \ol{e}_k:=
\begin{cases}
e_k\,\,&\text{in }\Omega \, ,\\
\mathrm{E}w\,\,&\text{in }\widetilde{\Omega} \setminus\Omega \, ,
\end{cases}
\quad \ol{p}_k:=
\begin{cases}
p_k\,\,&\text{in }\ol \Omega \, ,\\
0\,\,&\text{in }\widetilde{\Omega} \setminus\ol \Omega \, ,
\end{cases}
\end{equation*}
and
\begin{equation*}
    \ol u:=
    \begin{cases}
    u\,\,&\text{in }\Omega \, ,\\
    w\,\,&\text{in }\widetilde{\Omega} \setminus\Omega \, ,
    \end{cases}
    \quad \ol{e}:=
    \begin{cases}
    e\,\,&\text{in }\Omega \, ,\\
    \mathrm{E}w\,\,&\text{in }\widetilde{\Omega} \setminus\Omega \, ,
    \end{cases}
    \quad \ol{p}:=
    \begin{cases}
    p\,\,&\text{in }\ol \Omega \, ,\\
    0\,\,&\text{in }\widetilde{\Omega} \setminus\ol \Omega \, .
    \end{cases}
\end{equation*}
Moreover, given $\alpha_k$, $\alpha$ as in the statement, we let $\ol \alpha_k$ and $\ol \alpha$ be $W^{1,\gamma}$ extenstions of $\alpha_k$ and $\alpha$ to $\widetilde{\Omega}$, respectively. Then $\mathrm{E}\ol u_k=\ol e_k+ \ol p_k$ and $\mathrm{E}\ol u=\ol e+ \ol p$  as measures in $\M_b(\widetilde{\Omega};\MD)$, $\ol u_k\wtos \ol u$ in $BD(\widetilde{\Omega})$, $\ol e_k\to \ol e$ strongly in $L^2(\widetilde{\Omega};\Mnn)$, and
\begin{equation*}
\H(\alpha_k,p_k)= \int\limits_{\widetilde{\Omega}} V(\ol \alpha_k) \d |\ol p_k|\,, \quad  \H(\alpha,p)= \int\limits_{\widetilde{\Omega}} V(\ol \alpha) \d |\ol p| \, . 
\end{equation*}
(Notice 
that the formula above makes sense for the precise representatives of $\ol \alpha_k$ and $\ol \alpha$, but we did not write it explicitely not to overburden the notation.) 
 With a slight abuse of notation, in what follows we drop the notation $(\ol u_k, \ol e_k, \ol p_k)$, $(\ol u, \ol e, \ol p)$, $\ol \alpha_k$, $\ol \alpha$ for the extended functions and we consider the triples $(u_k,e_k,p_k)$, $(u,e,p)$ and the functions $\alpha_k$, $\alpha$ as already extended to~$\tilde{ \Omega}$ as described above. Moreover, we adapt the definition of admissible triples accordingly by putting
\begin{equation*}
    \begin{split}
        \mathcal{A}(w):=\{(u,e,p) \in  BD(\tilde \Omega)&  \x L^2(\tilde \Omega;\Mnn) \x \M_b(\tilde \Omega;\MnnD): \\ & \mathrm{E}u=e + p \,\text{ in }\tilde \Omega\, ,\, u = w \text{ in } \tilde \Omega \sm \ol \Omega \, , \, e = \EE w \text{ in } \tilde \Omega \sm \ol \Omega\}\, .
    \end{split}
\end{equation*}

%
We now show how to adapt the technique used in the proof of Proposition~\ref{2608161903} to the present setting, omitting some details when they are completely analogous to those in the proof of Proposition~\ref{2608161903}. Let us define,  for every direction $\xi\in \mathbb{S}^{n-1}$, every $\alpha\in   W^{1,\gamma}(\Omega)$,   every~$p$ such that $(u,e,p)\in \mathcal{A}(w)$, and every $A\subset \Omega$ open, 
 \begin{equation}\label{0109161819}
    \begin{split}
        \mathcal{G}_\xi(\alpha,p;A)&:=\int\limits_A V(\widetilde{\alpha}) \d|p\,\xi\cdot\xi| =\int\limits_A V(\widetilde{\alpha}) \d|(\mathrm{E}u - e)\,\xi\cdot\xi| \\
        & = \int \limits_{\Pi^\xi}  \int\limits_{A^\xi_y} V(\widetilde{\alpha}^\xi_y(t)) \d|\mathrm{D}\widehat{u}^\xi_y - (e\, \xi \cdot \xi)^\xi_y|(t)  \d \H^{n-1}(y) \\
        & = \int \limits_{\Pi^\xi}  \int\limits_{A^\xi_y} V(\widetilde{\alpha}^\xi_y(t)) \d|(p \, \xi \cdot \xi)^\xi_y|(t)  \d \H^{n-1}(y)\,.
    \end{split}
 \end{equation}
The functionals $\mathcal{G}_\xi$ will play the role of the functionals $\mathcal{F}_\xi$ defined in~\eqref{2608162103}. More precisely, we claim that for every $\alpha_k$, $\alpha \in   W^{1,\gamma}(\tilde \Omega)$,   $(u_k,e_k,p_k)$, $(u,e,p) \in \mathcal{A}(w)$ such that $\alpha_k \wto \alpha$ in $W^{1,\gamma}(\tilde \Omega)$, $u_k \wtos u$ in $BD(\tilde \Omega)$, and $e_k \to e$ strongly in $L^2(\tilde{\Omega};\Mnn)$ the following inequality holds true
\begin{equation} \label{eq:claim on Gxi}
    \mathcal{G}_\xi(\alpha,p;A) \leq \liminf_{k\to \infty} \mathcal{G}_\xi(\alpha_k,p_k;A) \, .
\end{equation}
To prove this, we start by extracting a (not relabeled) subsequence such that 
\begin{equation} \label{eq:convergences of slices}
    (\tilde \alpha_k)^\xi_y \to \tilde \alpha^\xi_y \, , \quad (\widehat u_k)^\xi_y \to \widehat u^\xi_y \, , \quad (e_k \, \xi \cdot \xi)^\xi_y \to (e \, \xi \cdot \xi)^\xi_y  \quad \text{ in } L^1(\tilde \Omega^\xi_y) \, .
\end{equation}
Let us fix $\epsilon > 0$. Since 
\begin{equation*}
    \liminf_{k \to \infty} \big[ \mathcal{G}_\xi(\alpha_k,p_k;A) + \epsilon |\EE u_k|(A) + \epsilon \| \nabla \alpha \|_{L^\gamma(A)} \big] < +\infty  \, ,
\end{equation*}
by Fatou's Lemma as in~\eqref{eq:liminf of section finite} we deduce that   for $\H^{n-1}$-a.e.\ $y \in \Pi^\xi$  
\begin{equation*}
    \liminf_{k \to \infty} \bigg[ \int\limits_{A^\xi_y} V((\widetilde{\alpha}_k)^\xi_y) \d|(p_k\, \xi \cdot \xi)^\xi_y|    + \epsilon |\mathrm{D} (\widehat{u}_k)^\xi_y|(A^\xi_y) + \epsilon \int \limits_{A^\xi_y} |\nabla (\tilde \alpha_k)^\xi_y |^\gamma \d t \bigg] < +\infty \, .
\end{equation*}
As in~\eqref{eq:liminf of section is limit}, we extract a subsequence $k_j$ (possibly depending on $u$) such that the liminf above is actually a limit. On this subsequence we deduce that 
\begin{equation*}
    (\tilde{\alpha}_{k_j})^\xi_y \to \tilde \alpha^\xi_y \quad \text{uniformly in } \tilde \Omega^\xi_y \, , \quad (\widehat u_{k_j})^\xi_y \wtos \widehat u^\xi_y \quad \text{ in } BV(\tilde{\Omega}^\xi_y) \, .
\end{equation*}
In particular, from~\eqref{eq:convergences of slices} we obtain that 
\begin{equation*}
    (p_k\, \xi \cdot \xi)^\xi_y = \mathrm{D}(\widehat{u}_{k_j})^\xi_y - (e_{k_j}\, \xi \cdot \xi)^\xi_y \wtos \mathrm{D}\widehat{u}^\xi_y - (e\, \xi \cdot \xi)^\xi_y    = (p\, \xi \cdot \xi)^\xi_y  \text{ in } \M_b(\tilde{\Omega}^\xi_y)  \, .
\end{equation*}
We stress that the strong convergence of $e_k$ to $e$ is crucial to deduce the weak* convergence above. An application of Theorem~\ref{teo:Res} yields 
\begin{equation*}
    \begin{split}
        & \int\limits_{A^\xi_y} V(\widetilde{\alpha}^\xi_y) \d|(p\, \xi \cdot \xi)^\xi_y|   \\
        & \leq \liminf_{j \to \infty} \int\limits_{A^\xi_y} V((\widetilde{\alpha}_{k_j})^\xi_y) \d|(p_{k_j}\, \xi \cdot \xi)^\xi_y|   \\
        & \leq \lim_{j \to \infty} \bigg[ \int\limits_{A^\xi_y} V((\widetilde{\alpha}_{k_j})^\xi_y) \d|(p_{k_j}\, \xi \cdot \xi)^\xi_y|  + \epsilon |\mathrm{D} (\widehat{u}_{k_j})^\xi_y|(A^\xi_y) + \epsilon \int \limits_{A^\xi_y} |\nabla (\tilde \alpha_{k_j})^\xi_y |^\gamma \d t \bigg]
    \end{split}
    \end{equation*}
Integrating with respect to $y \in \Pi^\xi$ and letting $\epsilon \to 0$ we conclude the proof of~\eqref{eq:claim on Gxi}.

With~\eqref{eq:claim on Gxi} at hand, the proof of the theorem follows the lines of the localisation argument already presented in the proof of Theorem~\ref{thm:lsc with Eu} with minor adaptations. Now, instead of~\eqref{2908162105}, we put 
\begin{equation*}
    \varphi_j^i(x):=V(\widetilde{\alpha}(x))\bigg|\frac{\mathrm{d} p}{\mathrm{d}  |p|}(x)\,\xi^i_j\cdot \xi^i_j\bigg|\quad\text{for every }j\in\N,\,i=1,\dots,n,\,\text{and }x\in\Omega\,,
\end{equation*}
and we use the fact that $\bigl|\frac{\mathrm{d}p}{\mathrm{d}|p|}(x)\bigr|=1$ for  $|p|$-a.e. $x \in \Omega$, instead of \eqref{2908162106}.

\end{proof}

\begin{remark}\label{rem:2504190838}
In order to prove the existence of a globally stable quasi-static evolution for a model of perfect plasticity and gradient damage with a term $\|\nabla \alpha\|_{L^\gamma}^\gamma$, $\gamma>1$ in the energy,
it would be enough to prove the lower semicontinuity of $\H$ when $u_k\wtos u$ in $BD(\Omega)$ and $e_k\wto e$ in $\Lnn$ (only weakly). The main difficulty in this case is that it is not true that for every $\xi \in \mathbb{S}^{n-1}$ there exists a subsequence $e_{k_j}$ such that \eqref{eq:convergences of slices} holds true.

Therefore a possible strategy for the existence proof would be to find an a priori bound on $e_k$ that guarantees the strong convergence in $\Lnn$. 
Since the elasticity tensor $\C(\alpha)$ is equicoercive with respect to $\alpha\in [0,1]$, 
the strong convergence for $e_k$ would follow for instance by an uniform bound for the stresses $\sigma_k=\C(\alpha_k)e_k$ in $H^1_{\mathrm{loc}}(\Omega;\Mnn)$. 
In the framework of perfect plasticity, without damage, an a priori bound of this type for the stresses is proven in \cite{BenFre96} and \cite{Dem09}.
\end{remark}

\begin{remark}
We remark that we have considered only measures $p$ with values in $\MD$, since this is the form used in perfect plasticity. Nonetheless it is possible to prove Theorem~\ref{thm:main} also for $p$ valued in $\Mnn$,  with no modifications in the argument. 
\end{remark}


\bigskip
\noindent {\bf Acknowledgements.}
V.C.\ has been supported by the Marie Sk\l odowska-Curie Standard European Fellowship No.\ 793018. G.O.\ has been supported by the Alexander von Humboldt Foundation.

\bigskip
\bibliography{bibliography}
\bibliographystyle{siam}

\end{document}